\documentclass[11pt]{article}

\usepackage{amssymb,amsmath,amsfonts,amsthm,comment}
\usepackage{latexsym}
\usepackage{graphics}
\usepackage{indentfirst}
\usepackage{hyperref}

\allowdisplaybreaks

\newtheorem{innercustomthm}{Theorem}
\newenvironment{customthm}[1]
  {\renewcommand\theinnercustomthm{#1}\innercustomthm}
  {\endinnercustomthm}
\newtheorem{innercustomlem}{Lemma}
\newenvironment{customlem}[1]
  {\renewcommand\theinnercustomlem{#1}\innercustomlem}
  {\endinnercustomlem}
\newtheorem{innercustomcor}{Corollary}
\newenvironment{customcor}[1]
  {\renewcommand\theinnercustomcor{#1}\innercustomcor}
  {\endinnercustomcor}
\newtheorem*{thm*}{Theorem}
\newtheorem{thm}{Theorem}
\newtheorem{lem}[thm]{Lemma}

\newtheorem{cor}[thm]{Corollary}
\newtheorem{conj}[thm]{Conjecture}

\newtheorem{ques}[thm]{Question}

\newcommand{\N}{\mathbb{N}}

\begin{document}

\title{On the Ohba Number and Generalized Ohba Numbers of Complete Bipartite Graphs}

\author{Kennedy Cano$^1$, Emily Gutknecht$^1$,  Gautham Kappaganthula$^1$, \\ George Miller$^1$, Jeffrey A. Mudrock$^2$, and Ezekiel Thornburgh$^1$}

\footnotetext[1]{Department of Mathematics, College of Lake County, Grayslake, IL 60030.}

\footnotetext[2]{Department of Mathematics and Statistics, University of South Alabama, Mobile, AL 36688.  E-mail:  {\tt {mudrock@southalabama.edu}}}

\maketitle

\begin{abstract}

We say that a graph $G$ is chromatic-choosable when its list chromatic number $\chi_{\ell}(G)$ is equal to its chromatic number $\chi(G)$.  Chromatic-choosability is a well-studied topic, and in fact, some of the most famous results and conjectures related to list coloring involve chromatic-choosability.  In 2002 Ohba showed that for any graph $G$ there is an $N \in \N$ such that the join of $G$ and a complete graph on at least $N$ vertices is chromatic-choosable.  The Ohba number of $G$ is the smallest such $N$.  In 2014, Noel suggested studying the Ohba number, $\tau_{0}(a,b)$, of complete bipartite graphs with partite sets of size $a$ and $b$.  In this paper we improve a 2009 result of Allagan by showing that $\tau_{0}(2,b) = \lfloor \sqrt{b} \rfloor - 1$ for all $b \geq 2$, and we show that for $a \geq 2$, $\tau_{0}(a,b) = \Omega( \sqrt{b} )$ as $b \rightarrow \infty$.  We also initiate the study of some relaxed versions of the Ohba number of a graph which we call generalized Ohba numbers.  We present some upper and lower bounds of generalized Ohba numbers of complete bipartite graphs while also posing some questions.

\medskip

\noindent {\bf Keywords.}  graph coloring, list coloring, Ohba number, chromatic-choosability

\noindent \textbf{Mathematics Subject Classification.} 05C15 

\end{abstract}

\section{Introduction}\label{intro}

In this paper all graphs are nonempty, finite, simple graphs unless otherwise noted.  Generally speaking we follow West~\cite{W01} for terminology and notation.  The set of natural numbers is $\N = \{1,2,3, \ldots \}$.  For $m, n \in \N$ and $m \leq n$, we write $[m : n]$ for the set $\{m, \ldots, n\}$. If $m > n$, then $[m : n] = \emptyset$. Also, we write $[n]$ for $[1:n]$. Additionally, we let $\binom{a}{b} = 0$ when $a<b$.  When $f:X \rightarrow Y$ is a function and $A \subseteq X$, we take $f(A)$ to be the set $\{f(a) : a \in A\}$.  If $G$ is a graph and $S \subseteq V(G)$, we write $G[S]$ for the subgraph of $G$ induced by $S$.  If $G$ and $H$ are vertex disjoint graphs, the \emph{join} of $G$ and $H$, denoted $G \vee H$, is the graph consisting of $G$, $H$, and additional edges added so that each vertex in $G$ is adjacent to each vertex in $H$.  Also, $K_{n,m}$ represents the complete bipartite graphs with partite sets of size $n$ and $m$.

\subsection{List Coloring and Chromatic-Choosability}

List coloring is a variation on the classical vertex coloring problem which was introduced in the 1970s independently by Vizing~\cite{V76} and Erd\H{o}s, Rubin, and Taylor~\cite{ET79}.  In the classical vertex coloring problem, we seek a \emph{proper $k$-coloring} of a graph $G$ which is a coloring of the vertices of $G$ with colors from $[k]$ so that adjacent vertices receive different colors. The \emph{chromatic number} of a graph, denoted $\chi(G)$, is the smallest $k$ such that $G$ has a proper $k$-coloring.  For list coloring, we associate a \emph{list assignment} $L$ with a graph $G$ which assigns to each vertex $v \in V(G)$ a list~\footnote{What is referred to in the literature as a list is actually a set.} of colors $L(v)$.  The graph $G$ is \emph{$L$-colorable} if there exists a proper coloring $f$ of $G$ such that $f(v) \in L(v)$ for each $v \in V(G)$ (we refer to $f$ as a \emph{proper $L$-coloring} of $G$).  A list assignment $L$ is called a \emph{$k$-assignment} for $G$ if $|L(v)|=k$ for each $v \in V(G)$.  The \emph{list chromatic number} of a graph $G$, denoted $\chi_\ell(G)$, is the smallest $k$ such that $G$ is $L$-colorable whenever $L$ is a $k$-assignment for $G$.  We say $G$ is \emph{$k$-choosable} if $k \geq \chi_\ell(G)$. 

It is easy to prove that for any graph $G$, $\chi(G) \leq \chi_\ell(G)$.  Erd\H{o}s, Rubin, and Taylor~\cite{ET79} observed that if $m = \binom{2k-1}{k}$, then $\chi_\ell(K_{m,m}) > k$.  The following related result is often attributed to Vizing~\cite{V76} or Erd\H{o}s, Rubin, and Taylor~\cite{ET79}, but it is best described as a folklore result. 

\begin{thm} \label{thm: listbipartite}
For $k \in \N$, $\chi_\ell(K_{k,t})=k+1$ if and only if $t \geq k^k$.
\end{thm}

Graphs that satisfy $\chi(G) = \chi_\ell(G)$ are known as \emph{chromatic-choosable}~\cite{O02}. Many classes of graphs have been conjectured to be chromatic-choosable. Perhaps the most well-known conjecture involving list coloring is about chromatic-choosability.  Indeed the famous List Coloring Conjecture states that every line graph of a loopless multigraph is chromatic-choosable. The List Coloring Conjecture was formulated independently by many different researchers (see~\cite{HC92}).  In 1995, Galvin~\cite{G95} showed that the List Coloring Conjecture holds for line graphs of bipartite multigraphs, and in 1996, Kahn~\cite{K96} proved an asymptotic version of the conjecture.

Due to its importance in chromatic graph theory, chromatic-choosability has received much attention in the literature (see e.g., \cite{KM18, KM19, PW03, ST09, TV96}).  A relatively recent development, which is important for this paper, is that in 2015 Noel, Reed, and Wu proved Ohba's Conjecture. 

\begin{thm} [\cite{NR15}] \label{thm: Ohba}
If $G$ is a graph satisfying $|V(G)| \leq 2 \chi(G) + 1$, then $G$ is chromatic-choosable.
\end{thm} 

\subsection{The Ohba Number and Generalized Ohba Number}

In 2015, Allagan and Johnson~\cite{AJ15} introduced the notion of \emph{Ohba number} of a graph. The Ohba number of a graph $G$ is 0 when $G$ is chromatic-choosable, and it is the smallest $n \in \N$ for which $\chi_{\ell}(K_n \vee G) = \chi(K_n \vee G)$ when $G$ is not chromatic-choosable.  The Ohba number of every graph exists~\footnote{This was first shown in 2002~\cite{O02} and can be proven without Theorem~\ref{thm: Ohba}.  Interestingly, the analogue of this result for correspondence coloring~\cite{BK17}, which is a generalization of list coloring, was proven in 2017.} by Theorem~\ref{thm: Ohba} and the fact that $\chi(K_n \vee G) = n+\chi(G)$.  For each $s \in \N$ we define the \emph{$s^{th}$ generalized Ohba number} of a graph $G$ to be the smallest integer $n$ such that $\chi_{\ell}(K_n \vee G) - \chi(K_n \vee G) \leq s$. In the case that $\chi_{\ell}(G) - \chi(G) \leq s$, we take the $s^{th}$ generalized Ohba number of $G$ to be 0.  In 2014 Noel suggested determining the Ohba number of complete bipartite graphs~\cite{N14} which served as the main motivation for this paper.

Suppose $a,b \in \N$ satisfy $a \leq b$.  We let $\tau_0(a,b)$ be the Ohba number of $K_{a,b}$, and we let $\tau_s(a,b)$ be the $s^{th}$ generalized Ohba number of $K_{a,b}$ for each $i \in \N$.  In particular, $\tau_0(a,b)$ is 0 if $K_{a,b}$ is chromatic-choosable; otherwise, $\tau_0(a,b)$ is the smallest $p \in \N$ such that $\chi_{\ell}(K_p \vee K_{a,b}) = p+2 = \chi(K_p \vee K_{a,b})$.  Theorem~\ref{thm: Ohba} allows us to easily deduce the following.

\begin{lem}\label{lem: problem13}
For any $a,b \in \N$ with $a \leq b$, $\tau_0(a,b) \leq \max \{0, a+b-5\}$.
\end{lem}

\begin{proof}
From Theorem~\ref{thm: Ohba}, $G$ is chromatic-choosable if $|V(G)| \leq 2\chi(G) + 1$. Suppose $a + b > 5$.  If $G = K_{a+b-5} \vee K_{a,b}$, then $\chi(G) = a + b - 3$ and $|V(G)| = 2a + 2b -5$. So, $2\chi(G) + 1 = 2a + 2b - 5 \geq |V(G)|$. Thus, $G$ is chromatic-choosable which means $\tau_0(a,b) \leq a + b - 5$ whenever $a + b > 5$.

Now suppose $a + b \leq 5$, and let $G=K_{a,b}$. Then $\chi(G) = 2$ and $2\chi(G) + 1 = 5 \geq |V(G)|$ which means $G$ is chromatic-choosable. Thus, $\tau_0(a,b) = 0$ whenever $a + b \leq 5$. So, in general, $\tau_0(a,b) \leq$max$\{0, a + b - 5 \}$.
\end{proof}

It is natural to ask whether the upper bound in Lemma~\ref{lem: problem13} can be improved significantly. Allagan showed the following when $a=2$.

\begin{thm}[\cite{A09}]\label{thm: Allagan}
For any $b \geq 5$
$$\left\lfloor \sqrt{b} \right\rfloor - 1 \leq \tau_0(2,b)\leq \left\lceil \frac{-7 + \sqrt{8b + 17}}{2}\right\rceil.$$
\end{thm}

Notice that Lemma~\ref{lem: problem13} only says $\tau_0(2,b) \leq b-3$; whereas, Theorem~\ref{thm: Allagan} tells us $\tau_0(2,b) = \Theta(\sqrt{b})$ as $b \rightarrow \infty$.  We will improve Theorem~\ref{thm: Allagan} below by showing that $\tau_{0}(2,b) = \lfloor \sqrt{b} \rfloor - 1$ for all $b \geq 2$.

\subsection{Outline of Paper and Open Questions}

Motivated by Noel's 2014 question~\cite{N14}, the overall goal of this paper is to find bounds on $\tau_s(a,b)$ where $a,b \in \N$ with $a \leq b$ and $s$ is a nonnegative integer.  In Section~\ref{Basics} we prove some fundamental facts about $\tau_s(a,b)$ that are used throughout the rest of the paper.  One important result, which has a simple proof, shows that when $a \leq s+2$, we can identify when $\tau_s(a,b) = 0$.

\begin{lem}\label{lem: closeaandi}
Suppose $a,b \in \N$ with $a \leq b$ and $s$ is a nonnegative integer. If $a \leq s+1$, then $\tau_s(a,b) = 0$. Moreover, if $a=s+2$, then $\tau_s(a,b)=0$ if and only if $b < (s+2)^{s+2}$.
\end{lem}  

Lemma~\ref{lem: closeaandi} shows that bounding $\tau_s(a,b)$ becomes interesting when $a \geq s+2$.  In Section~\ref{below} we establish some lower bounds on $\tau_s(a,b)$.  We begin Section~\ref{below} by focusing on $a = s+2$ and proving the following results. 

\begin{thm} \label{thm: its6not8}
Suppose $s$ and $l$ are nonnegative integers. Then $\tau_s(s+2,b) \geq l$ when $b \geq \binom{l+s+1}{l}(l+s+1)^{s+1}$.
\end{thm}

The proof of Theorem~\ref{thm: its6not8} for $l \geq 2$ involves the construction of an $(l+s+1)$-assignment $L$ for $G = K_{l-1} \vee K_{s+2,b}$ such that there is no proper $L$-coloring of $G$.  Hence, $\chi_{\ell}(G) - \chi(G) > s$. 

\begin{cor} \label{cor: its6not8}
Let $s$ be any nonnegative integer. If $b$ satisfies $b \geq s+2$ and $b \geq (s+1)^{2s+2}/(s+1)!$, then $\tau_s(s+2,b) \geq \lfloor ((s+1)!b)^{1/(2s+2)} \rfloor - s-1$. 
\end{cor}

We can view Corollary~\ref{cor: its6not8} as a generalization of the lower bound in Theorem~\ref{thm: Allagan} since it is easy to see that it implies  $\tau_{0}(2,b) \geq \lfloor \sqrt{b} \rfloor - 1$ for all $b \geq 2$.  We end Section~\ref{below} by establishing some lower bounds on $\tau_s(a,b)$ for any $a$ satisfying $a > s+2$.  This requires a more intricate construction than the proof of Theorem~\ref{thm: its6not8}, and  the formula for our bound that establishes a lower bound of $l$ depends on whether $(a-s-1)$ divides $(l+s+1)$.  

\begin{thm}\label{thm: r=0}
Suppose $s$, $l$, and $a$ are any nonnegative integers such that $l \geq 2$, $a > s+2,$ and $(a-s-1)$ divides $(l+s+1)$. Let $q=(l+s+1)/(a-s-1)$.
If $b \geq a$ and
$$b \geq  (l+s+1)^s\left(\binom{l+s+1}{l+1} + \binom{l+s+1}{l} q\right),$$
then $\tau_s(a,b) \geq l$.
\end{thm}

\begin{thm}\label{thm: improvedgeneral}
Suppose $s$, $l$, and $a$ are any nonnegative integers such that $l \geq 2$, $a > s+2,$ and $(a-s-1)$ does not divide $(l+s+1)$. Let $q=\lfloor(l+s+1)/(a-s-1)\rfloor$ and
$$u = ((1+q)(a-s-1)-(l+s+1))q.$$
If $b \geq a$ and
$$b \geq (l+s+1)^s\left(\binom{l+s+1}{l+1} + \binom{l+s+1}{l} (q+1) - \binom{u}{l}\right),$$
then $\tau_s(a,b) \geq l$. 
\end{thm}

It is easy to verify that Corollary~\ref{cor: its6not8} along with Theorems~\ref{thm: r=0} and~\ref{thm: improvedgeneral} imply that for fixed $a \geq s+2$, $\tau_s(a,b) = \Omega(b^{1/(2s+2)})$ as $b \rightarrow \infty$.  This observation is related to Question~\ref{ques: asym} below.

Finally, in Section~\ref{above} we focus on establishing upper bounds on $\tau_s(a,b)$ when $a = s+2$.  For our first result of Section~\ref{above} (Theorem~\ref{thm: half}), we prove a lemma that generalizes a theorem of Gravier, Maffray, and Mohar~\cite{GM03} which may be of independent interest (see Lemma~\ref{lem: upperinduction_general} below). Theorem~\ref{thm: half} shows that the bound in Theorem~\ref{thm: its6not8} is best possible when $s=0$.

\begin{thm}\label{thm: half}
Suppose $s$ is a nonnegative integer and $l \in \N$. Then $\tau_{s}(s+2,b) \leq l-1$ whenever $b \leq (l+s+1)^{s+2} -1$.
\end{thm}

Theorems~\ref{thm: its6not8} and~\ref{thm: half} allow us to find the exact value of $\tau_{s}(s+2,b)$ for some small values of $s$ and appropriate $b$.  For example, $\tau_1(3,b) = 2$ whenever $96 \leq b \leq 124$ and $\tau_2(4,b) = 2$ whenever $1250 \leq b \leq 1295$.  Most importantly, these Theorems allow us to improve Theorem~\ref{thm: Allagan} in the strongest possible sense.  

\begin{cor} \label{cor: improveAllagan}
For any $b \geq 2$, $\tau_0(2,b) = \lfloor\sqrt{b}\rfloor - 1.$
\end{cor}

We end Section~\ref{above} by developing some further ideas that allow us to determine precisely when $\tau_{s}(s+2,b)$ equals 1.  This result along with Lemma~\ref{lem: closeaandi} means we can fully characterize when $\tau_{s}(s+2,b)$ is 0 or 1.

\begin{thm}\label{cor: l=1}
For each nonnegative integer $s$, $\tau_{s}(s+2,b) = 1$ if and only if $(s+2)^{s+2} \leq b \leq \binom{s+3}{2}(s+3)^{s+1}-1$.
\end{thm}

With Theorem~\ref{cor: l=1} in mind, one can notice by plugging in $1$ for $l$ that Theorem~\ref{thm: half} is not tight.  We conjecture that Theorem~\ref{thm: half} is far from tight as $l$ increases, and in fact Theorem~\ref{thm: its6not8} gives the best possible result.

\begin{conj} \label{conj: tight}
Suppose $l \in \N$ and $s$ is a nonnegative integer.  Then, $\tau_{s}(s+2,b) = l$ if and only if $\binom{l+s+1}{l}(l+s+1)^{s+1} \leq b \leq \binom{l+s+2}{l+1}(l+s+2)^{s+1}-1$. 
\end{conj}  

It would be interesting to find upper bounds on $\tau_s(a,b)$ when $a > s+2$.  So, we propose this as a topic of future research.  More specifically, we wonder if the bounds in Theorems~\ref{thm: r=0} and~\ref{thm: improvedgeneral} are asymptotically best possible.  This leads to the following question.

\begin{ques} \label{ques: asym}
Suppose that $s$ is a nonnegative integer and $a \in \N$ satisfies $a \geq s+2$.  Is it the case that $\tau_s(a,b) = \Theta(b^{1/(2s+2)})$ as $b \rightarrow \infty$?
\end{ques}

Note that Corollary~\ref{cor: improveAllagan} implies that the answer to Question~\ref{ques: asym} is yes when $s=0$ and $a=2$.  Moreover, if Conjecture~\ref{conj: tight} is correct, the answer to Question~\ref{ques: asym} would be yes whenever $s$ is a nonnegative integer and $a = s+2$.

\section{Basic Results}\label{Basics}

In this section we prove some basic facts about $\tau_s(a,b)$.  We begin by justifying a property used frequently: for $a,b \in \N$ with $a \leq b$ and $s$ a nonnegative integer, if $\chi_{\ell}(K_p \vee K_{a,b}) - \chi(K_p \vee K_{a,b}) > s$, then $\tau_s(a,b) > p$.  Before we prove this, we need two lemmas.

\begin{lem}\label{lem: difference}
For any graph $G$, $\chi_{\ell}(K_1 \vee G) - \chi(K_1 \vee G) \leq \chi_{\ell}(G) - \chi(G)$.
\end{lem}

\begin{proof}
Suppose $\chi(G) = m$ and $d = \chi_{\ell}(G) - \chi(G)$. Let $H = K_1 \vee G$. Then $\chi(H) = m+1$ and $\chi_{\ell}(G) = m + d$.  Suppose $w$ is the vertex of the copy of $K_1$ used to form $H$. Suppose $L$ is an arbitrary $(m+d+1)$-assignment for $H$.  Color $w$ with some $c \in L(w)$. For each $v \in V(G)$ let $L'(v) = L(v) - \{c\}$. Clearly $|L'(v)| \geq m + d$. Since $\chi_{\ell}(G) = m + d$, there is a proper $L'$-coloring $f$ of $G$. Coloring the vertices of $G$ according to $f$ completes a proper $L$-coloring of $H$.  This means $\chi_{\ell}(H) \leq m+d+1$, and the result follows.
\end{proof}

\begin{lem}\label{lem: alsoinduction}
Suppose $a,b \in \N$ and both $r$ and $s$ are nonnegative integers such that $\chi_{\ell}(K_r \vee K_{a,b}) - \chi(K_r \vee K_{a,b}) \leq s$. Then for any integer $l \geq r$, $\chi_{\ell}(K_l \vee K_{a,b}) - \chi(K_l \vee K_{a,b}) \leq s$.
\end{lem}

\begin{proof}
The proof is by induction on $l$. The basis step follows from the fact that $\chi_{\ell}(K_r \vee K_{a,b}) - \chi(K_r \vee K_{a,b}) \leq s$. Assume $l > r$, and the desired statement holds for all integers less than $l$ and greater than or equal to $r$. By the inductive hypothesis, $\chi_{\ell}(K_{l-1} \vee K_{a,b}) - \chi(K_{l-1} \vee K_{a,b}) \leq s$. By Lemma~\ref{lem: difference}, $\chi_{\ell}(K_1 \vee K_{l-1} \vee K_{a,b}) - \chi(K_1 \vee K_{l-1} \vee K_{a,b}) \leq s$. Note that $K_1 \vee K_{l-1} \vee K_{a,b} = K_l \vee K_{a,b}$. Thus, the induction step is complete.
\end{proof}

\begin{lem}\label{lem: alsoobvious}
Suppose $p, a, b \in \N$ and $s$ is a nonnegative integer. If $\chi_{\ell}(K_p \vee K_{a,b}) - \chi(K_p \vee K_{a,b}) > s$, then $\tau_s(a,b) > p$.
\end{lem}

\begin{proof}
The contrapositive of Lemma~\ref{lem: alsoinduction} states that if there exists an integer $l \geq 0$ such that $\chi_{\ell}(K_l \vee K_{a,b}) - \chi(K_l \vee K_{a,b}) > s$, then for any $r$ satisfying $0 \leq r \leq l$, $\chi_{\ell}(K_r \vee K_{a,b}) - \chi(K_r \vee K_{a,b}) > s$. If $\chi_{\ell}(K_p \vee K_{a,b}) - \chi(K_p \vee K_{a,b}) > s$, there is no $r$ satisfying $0 \leq r \leq p$ such that $\chi_{\ell}(K_r \vee K_{a,b}) - \chi(K_r \vee K_{a,b}) \leq s$. Thus, $\tau_s(a,b) > p$. 
\end{proof}






Finally, we show that when $a \leq s+2$, we can identify when $\tau_s(a,b) = 0$.  Our proof makes use Theorem~\ref{thm: listbipartite} and the following well-known result: for $a,b \in \N$, $\chi_{\ell}(K_{a,b}) \leq a+1$.

\begin{customlem} {\bf \ref{lem: closeaandi}}
Suppose $a,b \in \N$ with $a \leq b$ and $s$ is a nonnegative integer. If $a \leq s+1$, then $\tau_s(a,b) = 0$. Moreover, if $a=s+2$, then $\tau_s(a,b)=0$ if and only if $b < (s+2)^{s+2}$.
\end{customlem}

\begin{proof}
Suppose $a \leq s + 1$. Then $\chi_{\ell}(K_{a,b}) \leq a + 1 \leq s + 2$. Consequently $\chi_{\ell}(K_{a,b}) - \chi(K_{a,b}) \leq s + 2 - 2 = s$, which means $\tau_s(a,b) = 0$.

Now, suppose $a=s+2$. If $b < (s+2)^{s+2}$, Theorem~\ref{thm: listbipartite} implies $\chi_{\ell}(K_{s + 2,b}) \leq s + 2$. Consequently $\chi_{\ell}(K_{s + 2,b}) - \chi(K_{s + 2,b}) \leq s + 2 - 2 = s$, which means $\tau_s(s + 2,b) = 0$.  Conversely, suppose $b \geq (s+2)^{s+2}$. Theorem~\ref{thm: listbipartite} implies $\chi_{\ell}(K_{s + 2,b}) = s + 3$ which means $\chi_{\ell}(K_{s + 2,b}) - \chi(K_{s + 2,b}) > s$. So, by Lemma~\ref{lem: alsoobvious} $\tau_s(s + 2,b) > 0$.
\end{proof}

\section{Bounding from Below} \label{below}

In this section we will introduce new lower bounds on $\tau_s(a,b)$.  We begin by proving Theorem~\ref{thm: its6not8} which we restate.

\begin{customthm} {\bf \ref{thm: its6not8}} 
Suppose $s$ and $l$ are nonnegative integers. Then $\tau_s(s+2,b) \geq l$ when $b \geq \binom{l+s+1}{l}(l+s+1)^{s+1}$.
\end{customthm}

\begin{proof}
The result is obvious when $l=0$.  Note that by Lemma~\ref{lem: closeaandi} the result holds when $l=1$.  So, suppose $l\geq2$ and $b \geq \binom{l+s+1}{l}(l+s+1)^{s+1}$. Let $G = K_{l-1} \vee K_{s+2, b}$. We will show that $\chi_{\ell}(G) > l+s+1$ which by Lemma~\ref{lem: alsoobvious} will imply that $\tau_s(s+2, b)>l-1$ as desired. 

Suppose the vertex set of the copy of $K_{l-1}$ used to form $G$ is $W = \{w_1,\dots,w_{l-1}\}$. Also, suppose the bipartition of the copy of $K_{s+2,b}$ used to form $G$ is $X = \{x_1,\ldots, x_{s+2}\}$, $Y = \{y_1,\ldots, y_b\}$. Consider the $(l+s+1)$-assignment for $G$ defined as follows. Let $L(v) = [l+s+1]$ whenever $v \in W \bigcup \{x_{s+2}\}$. Let $L(x_j) = [j(l+s+1)+1 : (j+1)(l+s+1)]$ for each $j \in [s+1]$. Let $A$ be the set of all $l$-element subsets of $[l+s+1]$, and let 
$$C = \left\{ I \cup \bigcup_{j=1}^{s+1}\{c_j\} : I \in A, c_j \in L(x_j) \text{ for each } j \in [s+1]\right\}.$$

Note that $|C| = \binom{l+s+1}{l}(l+s+1)^{s+1}$, and let $c = |C|$. Name the elements of $C$ so that $C = \{C_1,\ldots, C_c\}$. Let $L(y_i) = C_i$ for each $i \in [c]$. Finally, let $L(y_j) = [l+s+1]$ whenever $c < j \leq b$.

For the sake of contradiction, suppose that $f$ is a proper $L$-coloring of $G$. Since $W \bigcup \{x_{s+2}\}$ is a clique in $G$, we know $f(W \bigcup \{x_{s+2}\}) \in A$. Consequently, $f(W \bigcup X) \in C$ which means there is an $r \in [c]$ such that $f(W \bigcup X) = C_r$. Notice $y_r$ is adjacent to all the vertices in $W \bigcup X$. Since $f(y_r) \in L(y_r) = C_r$, $f$ is not proper which is a contradiction.
\end{proof}

\begin{customcor} {\bf \ref{cor: its6not8}}
Let $s$ be any nonnegative integer. If $b$ satisfies $b \geq s+2$ and $b \geq (s+1)^{2s+2}/(s+1)!$, then $\tau_s(s+2,b) \geq \lfloor ((s+1)!b)^{1/(2s+2)} \rfloor - s-1$. 
\end{customcor}

\begin{proof}
Note that $\lfloor \sqrt[2s+2]{(s+1)!b} \rfloor - s-1  \geq 0$. Let $l = \lfloor \sqrt[2s+2]{(s+1)!b} \rfloor - s-1$. Notice
$$b = \frac{\left(\sqrt[2s+2]{(s+1)!b}\right)^{2s+2}}{(s+1)!} \geq \frac{(l+s+1)^{2s+2}}{(s+1)!} \geq \binom{l+s+1}{l}(l+s+1)^{s+1}.$$
 By Theorem~\ref{thm: its6not8}, $\tau_s(s+2,b) \geq \lfloor \sqrt[2s+2]{(s+1)!b} \rfloor - s-1$.
\end{proof}

We can view Corollary~\ref{cor: its6not8} as a generalization of the lower bound in Theorem~\ref{thm: Allagan} since it gives us the lower bound in Theorem~\ref{thm: Allagan} when $s=0$.  We now establish some lower bounds on $\tau_s(a,b)$ for any $a$ satisfying $a > s+2$.  Interestingly, the formula for our bound that establishes a lower bound of $l$ depends on whether $(a-s-1)$ divides $(l+s+1)$. 

\begin{customthm} {\bf \ref{thm: r=0}}
Suppose $s$, $l$, and $a$ are any nonnegative integers such that $l \geq 2$, $a > s+2,$ and $(a-s-1)$ divides $(l+s+1)$. Let $q=(l+s+1)/(a-s-1)$.
If $b \geq a$ and
$$b \geq  (l+s+1)^s\left(\binom{l+s+1}{l+1} + \binom{l+s+1}{l} q\right),$$
then $\tau_s(a,b) \geq l$.
\end{customthm}

\begin{proof}
Let $G = K_{l-1} \vee K_{a,b}$. We will show that $\chi_{\ell}(G) > l+s+1$, which by Lemma~\ref{lem: alsoobvious} will imply $\tau_s(a,b) > l-1$ as desired.

Suppose the vertex set of the copy of $K_{l-1}$ used to form $G$ is $W = \{w_1,\ldots, w_{l-1}\}$. Also, suppose the bipartition of the copy of $K_{a,b}$ used to form $G$ is $X = \{x_1,\ldots,x_a\}$, $Y=\{y_1, \ldots, y_b\}$. Consider the $(l+s+1)$-assignment for $G$ defined as follows. Let $L(v) = [l+s+1]$ whenever $v \in W \bigcup \{x_a\}$. For each $d \in [a-s: a-1]$, let $L(x_{d}) = [(d-(a-s-1))(l+s+1)+1:(d-(a-s-1)+1)(l+s+1)]$. For every $j \in [a-s-1]$, let $P_j = [(j-1)q + 1: jq]$. Note that $\{P_1, \ldots, P_{a-s-1}\}$ is a partition of $[l+s+1]$. For every $j \in [a-s-1]$ let $L(x_j) = ([l+s+1] - P_j) \bigcup [(s+1)(l+s+1)+1:(s+1)(l+s+1)+ q]$. Let $A_{l+1}$ be the set of all $(l+1)$-element subsets of $[l+s+1]$. Let $A_{l}$ be the set of all $l$-element subsets of $[l+s+1]$. Notice $|A_{l+1}| = \binom{l+s+1}{l+1}$ and $|A_{l}| = \binom{l+s+1}{l}$.  Let $\mathcal{J} = \{\bigcup_{j=1}^{s}\{m_j\} : \text{ for each } j\in[s], m_j\in [j(l+s+1)+1 : (j+1)(l+s+1)]\}$, and take $\mathcal{J}=\emptyset$ when $s=0$. Let

$$T= \{I \cup J \cup \{i\}: I \in A_l, J \in \mathcal{J}, \\ i \in [(s+1)(l+s+1)+1: (s+1)(l+s+1) + q]\},$$
$$Z = \{I \cup J : I \in A_{l+1}, J \in \mathcal{J}\}, \text{ and } C = T \cup Z.$$

Note that $|C| = (l+s+1)^s\left(\binom{l+s+1}{l+1} + \binom{l+s+1}{l} q \right)$. Let $c = |C|$. Name the elements of $C$ so that $C = \{C_1,\ldots, C_c\}$. Let $L(y_i) = C_i$ for each $i \in [c]$. Finally, let $L(y_j) = [l+s+1]$ whenever $j \in [c+1:b]$.

For the sake of contradiction, suppose that $f$ is a proper $L$-coloring of $G$. Since $W \bigcup \{x_a\}$ is a clique in $G$, we know $f(W \bigcup \{x_a\}) \in A_{l}$. Let $D = \bigcup_{d=a-s}^{a-1}L(x_d)$. Let $u=(s+1)(l+s+1) + q$. We claim that $f(\{x_1,\ldots, x_{a-s-1}\})$ contains an element of $[u] - (f(W \bigcup \{x_a\}) \bigcup D)$. To see why, assume $f(\{x_1,\ldots, x_{a-s-1}\}) \subseteq (f(W \bigcup \{x_a\}) \bigcup D)$. Since $D \cap L(x_j) = \emptyset$ for each $j \in [a-s-1]$, $f(x_j) = f(x_a)$ for each $j \in [a-s-1]$. Since $f(x_a) \in [l+s+1]$ there exists a $t \in [a-s-1]$ such that $f(x_a) \in P_t$. This means $f(x_a) \notin L(x_t)$ which contradicts $f(x_t) = f(x_a)$. Since $\{f(x_{a-s}),\ldots, f(x_{a-1})\} \in \mathcal{J}$,  there is a $C_p \in C$ such that $C_p \subseteq f(W \bigcup X)$.
Note that $f(y_p) \in C_p$ and $y_p$ is adjacent to every element in $W\bigcup X$. This contradicts the fact that $f$ is a proper $L$-coloring.

\end{proof}

\begin{customthm} {\bf \ref{thm: improvedgeneral}}
Suppose $s$, $l$, and $a$ are any nonnegative integers such that $l \geq 2$, $a > s+2,$ and $(a-s-1)$ does not divide $(l+s+1)$. Let $q=\lfloor(l+s+1)/(a-s-1)\rfloor$ and
$$u = ((1+q)(a-s-1)-(l+s+1))q.$$
If $b \geq a$ and
$$b \geq (l+s+1)^s\left(\binom{l+s+1}{l+1} + \binom{l+s+1}{l} (q+1) - \binom{u}{l}\right),$$
then $\tau_s(a,b) \geq l$. 
\end{customthm}

\begin{proof}
Let $G = K_{l-1} \vee K_{a,b}$. We will show that $\chi_{\ell}(G) > l+s+1$, which by Lemma~\ref{lem: alsoobvious} will imply $\tau_s(a,b) > l-1$ as desired.

Suppose the vertex set of the copy of $K_{l-1}$ used to form $G$ is $W = \{w_1,\ldots,w_{l-1}\}$. Also, suppose the bipartition of the copy of $K_{a,b}$ used to form $G$ is $X = \{x_1,\ldots,x_a\}$, $Y=\{y_1,\ldots,y_b\}$. Consider the $(l+s+1)$-assignment for $G$ defined as follows. Let $L(v) = [l+s+1]$ whenever $v \in W \bigcup \{x_a\}$. For each $d \in [a-s: a-1]$, let $L(x_{d}) = [(d-(a-s-1))(l+s+1)+1:(d-(a-s-1)+1)(l+s+1)]$. Suppose $r$ is a nonnegative integer such that $ 0 < r \leq a-s-2$ and $l + s+1 = (a-s-1)q +r$. Then for each $i \in [r]$ let $P_i = [(i-1)(q+1) + 1: i(q+1)]$. Then for every $j \in [a-s-1-r]$, let $P_{r+j} = [r(q+1) + (j-1)q + 1: r(q+1) + jq]$. In the case where $q=0$, we take $P_{r+j}$ to be the empty set. Note that $\{P_1,\ldots, P_{a-s-1}\}$ is a partition of $[l+s+1]$. For each $i \in [r]$ let $L(x_i) = ([l+s+1] - P_i) \bigcup [(s+1)(l+s+1)+1:(s+1)(l+s+1) + (q+1)]$. For every $j \in [a-s-1-r]$ let $L(x_{r+j}) = ([l+s+1] - P_{r+j}) \bigcup [(s+1)(l+s+1)+1:(s+1)(l+s+1)+ q]$. Let $A_{l+1}$ be the set of all $(l+1)$-element subsets of $[l+s+1]$. Let $A_{l}$ be the set of all $l$-element subsets of $[l+s+1]$. Notice $|A_{l+1}| = \binom{l+s+1}{l+1}$ and $|A_{l}| = \binom{l+s+1}{l}$. Let $\mathcal{J} = \{\bigcup_{j=1}^{s}\{m_j\} : \text{ for each } j\in[s], m_j\in [j(l+s+1)+1:(j+1)(l+s+1)]\}$, and take $\mathcal{J}=\emptyset$ when $s=0$. Let $h = (s+1)(l+s+1)$,

$$T_1 = \left\{I \cup \{h+(q+1)\}: I \in A_l, I \subseteq \bigcup_{i=r+1}^{a-s-1}P_i \right\}, \;\; \text{ and }$$
$$T_2 =\left\{I \cup \{i\}: I \in A_l, i \in [h+1: h + q+1] \right\}.$$
Notice $u$ is the cardinality of $\bigcup_{i=r+1}^{a-s-1}P_i$. Furthermore, if $u<l$, then $T_1=\emptyset$. Also, $|T_1|=\binom{u}{l}$ and $|T_2 - T_1| = \binom{l+s+1}{l} (q+1) - \binom{u}{l}$. 
Let
$$T= \{I \cup J: I \in (T_2-T_1), J \in \mathcal{J} \}, Z = \{I \cup J : I \in A_{l+1}, J \in \mathcal{J}\}, \text{ and } C = T \cup Z.$$

Note that $|C| = (l+s+1)^s\left(\binom{l+s+1}{l+1} + \binom{l+s+1}{l} (q+1) - \binom{u}{l}\right)$. Let $c = |C|$. Name the elements of $C$ so that $C = \{C_1,\ldots,C_c\}$. Let $L(y_i) = C_i$ for each $i \in [c]$. Finally, let $L(y_j) = [l+s+1]$ whenever $j \in [c+1:b]$.

For the sake of contradiction, suppose that $f$ is a proper $L$-coloring of $G$. Since $W \bigcup \{x_a\}$ is a clique in $G$, we know $f(W \bigcup \{x_a\}) \in A_{l}$. We claim that $f(\{x_1 ,\ldots, x_{a-s-1}\})$ contains an element of $[h+1: h+ q+1]\cup([l+s+1]-f(W \bigcup \{x_a\}))$. To see why, assume $f(\{x_1,\ldots, x_{a-s-1}\}) \subseteq (f(W \bigcup \{x_a\}) \bigcup [l+s+2:h])$. Since $[l+s+2:h] \cap L(x_j) = \emptyset$ for each $j \in [a-s-1]$, $f(x_j) = f(x_a)$ for each $j \in [a-s-1]$. Since $f(x_a) \in [l+s+1]$ there exists a $t \in [a-s-1]$ such that $f(x_a) \in P_t$. This means $f(x_a) \notin L(x_t)$ which contradicts $f(x_t) = f(x_a)$. 

Additionally, if $f(W \bigcup \{x_a\}) \subseteq \bigcup_{i=r+1}^{a-s-1}P_i$, we claim $f(\{x_1,\ldots,x_{a-s-1}\}) \neq \{h+q+1,f(x_a)\}$. For the sake of contradiction, suppose that $f(\{x_1,\ldots,x_{a-s-1}\}) = \{h+q+1,f(x_a)\}$. Since $f(x_a) \in  \bigcup_{i=r+1}^{a-s-1} P_i$, there exists a $z \in [r+1:a-s-1]$ such that $ \{f(x_a)\} \in P_z$. So $L(x_z)\cap\{h+q+1,f(x_a)\}=\emptyset$, which is a contradiction. Thus, $f(W \cup \{x_1,\ldots,x_{a-s-1}\} \cup \{x_a\}) \notin T_1$. Consequently, there exists an $S \in T_2 - T_1$ or an $S \in A_{l+1}$ such that $S \subseteq f(W \cup \{x_1,\ldots,x_{a-s-1}\} \cup \{x_a\})$.

Finally, when $s\in \N$, $f(\{x_{a-s},\ldots,x_{a-1}\}) \in \mathcal{J}$. So, there is a $C_p \in C$ such that $C_p \subseteq f(W \bigcup X)$.
Note that $f(y_p) \in C_p$ and $y_p$ is adjacent to every element in $W\bigcup X$. This contradicts the fact that $f$ is a proper $L$-coloring.
\end{proof}

It is easy to verify that Corollary~\ref{cor: its6not8} along with Theorems~\ref{thm: r=0} and~\ref{thm: improvedgeneral} imply that for fixed $a \geq s+2$, $\tau_s(a,b) = \Omega(b^{1/(2s+2)})$ as $b \rightarrow \infty$.  So, answering Question~\ref{ques: asym} in the affirmative requires establishing an upper bound on $\tau_s(a,b)$ when $b$ is large and $a \geq s+2$. 





\section{Bounding from Above} \label{above}
 
In this section we are primarily working with graphs of the form $G = K_l \vee K_{a,b}$ with $2 \leq a \leq b$.
When examining $G$ we are focusing on when the graph is $(l+a)$-choosable since $(l+a)$-choosability implies $\tau_{a-2}(a,b) \leq l$.
We will begin with a lemma that makes an important observation about how list overlap impacts the existence of a proper coloring.

\begin{lem}\label{lem: non-empty_General}
Suppose $l,b,s\in \N$, $s \geq 2$, and $G = K_l\vee K_{s,b}$.
Suppose the partite set of size $s$ from the copy of $K_{s,b}$ used to form $G$ is $\{x_1, \ldots, x_{s}\}$. Suppose $L$ is a $(l+s)$-assignment such that there exists $i,j \in [s]$ with $i \neq j$ and $L(x_i) \cap L(x_j) \neq \emptyset$.
Then there is a proper $L$-coloring of $G$. 
\end{lem}

\begin{proof}
Let $\{y_1, \ldots, y_b\}$ be the partite set of size $b$ used to form the copy of $G = K_l\vee K_{s,b}$.
Suppose the vertex set of the copy of $K_{l}$ used to form $G$ is $W = \{w_1,\ldots,w_{l}\}$. Let $X = \{x_1, \ldots, x_s\}$ and $Y =\{y_1,\ldots, y_b\}$. We know there is a $c \in L(x_i) \cap L(x_j)$. In order to construct a proper $L$-coloring, begin by coloring $x_i$ and $x_j$ with $c$.
Let $L^\prime(w_m) = L(w_m)- \{c\}$ for each $m\in[l]$, and $L'(x_q) = L(x_q)$ for each $q \in[s]$.
Note $|L^\prime(v)| \geq l+s-1$ for each $v\in(W \cup X)$. 
Then we can greedily construct a proper $L^\prime$-coloring, $g$, of $G[W \cup (X-\{x_i,x_j\})]$. 
Suppose the range of $g$ is $R$. 
Color the vertices of $W \cup (X-\{x_i,x_j\})$ according to $g$.
Let $L^{\prime\prime}(y_p) = L(y_p) - (\{c\} \cup R)$ for each $p\in[b]$. 
Note $|L^{\prime\prime}(y_p)| \geq 1$ for each $p \in [b]$, as $|\{c\} \cup R| \leq l+s-1$.
Finally color $y_p$ with any color from $L^{\prime\prime}(y_p)$ for each $p\in[b].$
This completes a proper $L$-coloring of $G$.
\end{proof}

If $G$ is a graph and $L$ is a list assignment for $G$, let $$\mathcal{R}(G,L) = \{f(V(G)): \text{ $f$ is a proper $L$-coloring of $G$}\}.$$ 
Notice by Lemma~\ref{lem: non-empty_General} that if $G = K_l\vee K_{s,b}$, the partite set of size $s$ of the copy of $K_{s,b}$ used to form $G$ is $\{x_1, \ldots, x_{s}\}$, and $L$ is an $(l+s)$-assignment, then it is only possible that $G$ is not $L$-colorable if $L(x_1), \ldots, L(x_s)$ are pairwise disjoint.  We will see that when $L(x_1), \ldots, L(x_s)$ are pairwise disjoint and $W$ is the vertex set of the copy of $K_l$ used to form $G$, if $|\mathcal{R}(G[W \cup X],L')| > b$ where $L'$ is $L$ with domain restricted to $W \cup X$, then there exists a proper $L$-coloring of $G$.  In the next two lemmas we work towards an upper bound on $|\mathcal{R}(G[W \cup X],L')|$.

\begin{lem}\label{lem: weighted}
Suppose $k \geq 2$ and $d_1, \ldots, d_k$ are positive integers with $d_1 \leq d_i$ for each $i \in [k]$.
Let $l$ be a nonnegative integer such that $l \leq d_1 -1$, and let $A_l = \{(a_1, \ldots, a_k) : \text{$a_i$ is a nonnegative integer for each $i \in [k]$ and $\sum_{i=1}^{k}a_i = l$}\}$.
Then, $$(d_1-l) \prod_{i=2}^{k}d_i \leq \prod_{i=1}^{k}(d_i-a_i)$$ for each $(a_1, \ldots, a_k)\in A_l$.
\end{lem}

\begin{proof}
Suppose $f_l : A_l \rightarrow \N$ is given by $f_l(a_1, \ldots, a_k) = \prod_{i=1}^{k}(d_i-a_i)$.
From among all the elements of $A_l$ that make $f_l$ as small as possible, choose one, say $(b_1,\ldots,b_k)$, so that the sum of the second through the $k^{th}$ coordinate is as small as possible. If $\sum_{i=2}^{k}b_i = 0$ then we are done. So suppose $\sum_{i=2}^{k}b_i > 0$.

Suppose $j \in [2:k]$ has the property that $b_j > 0$.
Let $c_1 = b_1+b_j$, $c_j = 0$, and $c_i = b_i$ for each $i \in([k]-\{1,j\})$
Note $(c_1,\ldots, c_k) \in A_l$, and by the choice of $(b_1, \ldots, b_k)$ we know that $f_l(c_1, \ldots, c_k) > f_l(b_1, \ldots, b_k)$.
This means $(d_1-c_1)(d_j-c_j) > (d_1-b_1)(d_j-b_j)$ which implies $(d_1-b_1-b_j)(d_j) > (d_1-b_1)(d_j-b_j)$ which implies $d_1-d_j > b_1$. Thus, $0 \geq d_1-d_j > b_1 \geq 0$ which is a contradiction. So $\sum_{i=2}^{k}b_i = 0$.
\end{proof}

We are now ready to prove a lemma that gives a lower bound on $|\mathcal{R}(G[W \cup X],L')|$. The lemma generalizes a theorem of Gravier, Maffray, and Mohar \cite{GM03}~\footnote{When $s=2$ Lemma~\ref{lem: upperinduction_general} reduces to Theorem~1 in~\cite{GM03}.}.

\begin{lem}\label{lem: upperinduction_general}
Suppose $G$ is an $n$-vertex graph with $n \geq 2$. 
Suppose $\{x_i: i \in [s]\}$ is a set of vertices in $G$ where $2 \leq s \leq n$.
Suppose $L$ is a list assignment for $G$ with $|L(v)| \geq n$ for each $v \in V(G)$. Also, suppose $L(x_1), \ldots, L(x_s)$ are pairwise disjoint.
Then $$|\{f(V(G)): \text{ $f$ is a proper $L$-coloring of $G$}\}| \geq \prod_{i=1}^s|L(x_i)|.$$
\end{lem}

\begin{proof}
Fix an $s$ with $s \geq 2$.
The proof is by induction on $n$. The result is obvious when $n = s$. So, suppose $n \geq s+1$ and the desired result holds for all positive integers less than $n$ and at least $s$.  For each $i \in [s]$ let $n_i = |L(x_i)|$. Suppose, $z\in V(G)-\{x_i: i \in [s] \}$ and $1 \in L(z)$.
Since $L(x_1), \ldots, L(x_s)$ are pairwise disjoint, we may suppose $1 \notin L(x_i)$ whenever $i \in [2:s]$.
Let $$A = \{f(V(G)): \text{$f$ is a proper $L$-coloring of $G$ with $f(z) = 1$}\}, $$ $$B = \{f(V(G)):\text{$f$ is a proper $L$-coloring of $G$ with }  1 \notin f(V(G))\},$$ 
and $C = \mathcal{R}(G,L)$.  Clearly, $A \cup B \subseteq C$ and $A \cap B = \emptyset$. Thus, $|C| \geq |A|+|B|$. 
Next, we will bound $|A|$ and $|B|$ from below. Let $G^\prime = G -{z}$. Note $G^\prime$ has $n-1$ vertices and $\{x_1, \ldots, x_s\}$ is an independent set in $G^\prime$. Let $L^\prime(v) = L(v) - \{1\}$ for each $v \in V(G')$. Notice that $|L^\prime(v)| \geq n-1$ for each $v \in V(G^\prime)$. Clearly $L'(x_1), \ldots, L'(x_s)$ are pairwise disjoint. Let $D = \{f(V(G')): \text{$f$ is a proper $L'$-coloring of $G'$}\}$. By the induction hypothesis, $|D| \geq (n_1 - 1)\prod_{i=2}^{s}n_i$. Consider the function $T: D \rightarrow A$ given by $T(R) = R \cup \{1\}$. It is easy to verify that $T$ is a bijection; therefore, $|D| = |A|$.

Now we will bound $|B|$ from below. Let $L''(v) = L(v) - \{1\}$ for each $v \in V(G)$. Now greedily color each $v \in (V(G)- \{x_1,\ldots,x_s\})$ with an element from $L''(v)$ (such a coloring exists since we are coloring $n-s$ vertices and $|L''(v)| \geq n-1$ for each $v \in V(G)$). Suppose $M$ consists of the colors used for this greedy coloring.
Let $L'''(x_i) = L''(x_i) - M$. Consider the function $T': \prod_{i=1}^s L'''(x_i) \rightarrow B$ given by $T'(b_1,\ldots,b_s) = M \cup \{b_1,\ldots, b_s\}$.
It is easy to verify that $T'$ is a injection; therefore, $|\prod_{i=1}^s L'''(x_i)| \leq |B|$.
Since $L'''(x_1),\ldots, L'''(x_s)$ are pairwise disjoint, $|\prod_{i=1}^s L'''(x_i)| = \prod_{i=1}^s |L'''(x_i)|$. 
Also since $|M| \leq n-s$, there exists nonnegative integers $a_1, \ldots, a_s$ satifying $\sum_{i=1}^{s}a_i \leq n-s$ such that $|L'''(x_1)| \geq n_1-1-a_1$ and $|L'''(x_i)| \geq n_i-a_i$ whenever $i \in [2:s]$.
Thus $\prod_{i=1}^s |L'''(x_i)| \geq (n_1-1-a_1)\prod_{i=2}^{s}(n_i-a_i)$.
We will show that $(n_1-1-a_1)\prod_{i=2}^{s}(n_i-a_i)\geq \prod_{i=2}^{s}n_i$ by considering two cases: (1) $n_1 \leq n_i$ for all $i \in [s]$ and (2) $n_i < n_1$ for some $i \in [2:s]$.

For (1) note that Lemma~\ref{lem: weighted}, $\sum_{i=1}^{s}a_i \leq n-s$, and $n_1 \geq n$ imply that 
\begin{align*}
(n_1-1-a_1)\prod_{i=2}^{s}(n_i-a_i) \geq \left(n_1-1-\sum_{i=1}^{s}a_i\right)\prod_{i=2}^{s}n_i &\geq (n_1-1-n+s)\prod_{i=2}^{s}n_i \\
&\geq (s-1)\prod_{i=2}^{s}n_i \\
&\geq \prod_{i=2}^{s}n_i.
\end{align*}  

For (2) without loss of generality assume $n_2 \leq n_i$ for each $i \in [s]$.
Lemma~\ref{lem: weighted}, $\sum_{i=1}^{s}a_i \leq n-s$, $n_1 > n_2$, and $n_2 \geq n$ imply that~\footnote{We take $\prod_{i=3}^{s}n_i$ to be 1 when $s = 2$.} 

\begin{align*}
(n_1-1-a_1)\prod_{i=2}^{s}(n_i-a_i) &\geq (n_1-1)\left(n_2-\sum_{i=1}^{s}a_i\right)\prod_{i=3}^{s}n_i \\
&\geq (n_1-1)(n_2-n+s)\prod_{i=3}^{s}n_i \\
&\geq n_2(n_2-n+s)\prod_{i=3}^{s}n_i \geq \prod_{i=2}^{s}n_i.
\end{align*}

Consequently, $|B| \geq \prod_{i=2}^{s}n_i$. Thus, $|C| \geq |A|+|B| \geq (n_1 - 1)\prod_{i=2}^{s}n_i + \prod_{i=2}^{s}n_i = \prod_{i=1}^{s}n_i$ which completes the induction step.
\end{proof}

Having established our lemmas, we are ready to prove Theorem~\ref{thm: half} which we restate.

\begin{customthm} {\bf \ref{thm: half}}
Suppose $s$ is a positive integer with $s \geq 2$ and $l$ is a nonnegative integer. Then $\tau_{s-2}(s,b) \leq l$ whenever $b \leq (l+s)^s -1$.
\end{customthm}

\begin{proof}
Notice that when $l=0$ the result follows from Lemma~\ref{lem: closeaandi}. 
So, suppose $l$ is positive.
Let $G = K_l \vee K_{s,b}$.
To show that $\tau_{s-2}(s,b) \leq l$, we will show that $\chi_{\ell}(G) -  \chi(G) \leq s-2$ which is equivalent to $\chi_{\ell}(G) \leq l+s$.

Suppose the vertex set of the copy of $K_l$ used to construct $G$ is $W = \{w_1, \ldots ,w_l\}$. 
Furthermore, the copy of $K_{s,b}$ used to construct $G$ has bipartition $X = \{x_1,\ldots,x_s\}, Y = \{y_1, \ldots , y_b\}$.  
Suppose $L$ is an arbitrary $(l+s)$-assignment for $G$. 
We must show that there exists a proper $L$-coloring of $G$.
Note that if there exists $i,j \in [s]$ with $i \neq j$ and $L(x_i) \cap L(x_j) \neq \emptyset$ the existence of a proper $L$-coloring of $G$ follows from Lemma~\ref{lem: non-empty_General}. So we may assume that $L(x_1), \ldots, L(x_s)$ are pairwise disjoint.

Let $L'$ be the function $L$ with domain restricted to $W \cup X$.
By Lemma~\ref{lem: upperinduction_general} we know that 
$$|\{f(W \cup X):\text{$f$ is a proper $L'$-coloring of $G[W \cup X]$}\}| \geq (l+s)^s.$$
Since $b \leq (l+s)^s-1$, the above inequality implies that there is a proper $L'$-coloring $g$ of $G[W \cup X]$ such that $L(y_i) \neq g(W \cup X)$ for each $i \in [b]$. Color the vertices in $W \cup X$ according to $g$. 
Let $L''(y_i) = L(y_i)-g(W \cup X)$ for each $i \in [b]$. Clearly $|L''(y_i)| \geq 1$ for each $i \in [b]$.
Finally, color $y_i$ with any color from $L^{\prime\prime}(y_i)$ for each $i\in[b]$.
This completes a proper $L$-coloring of $G$.
\end{proof}

We can now use Corollary~\ref{cor: its6not8} and Theorem~\ref{thm: half} to prove the following.

\begin{cor}\label{thm: decent_sandwich}
Suppose $s$ is an integer such that $s \geq 2$ and b is an integer such that $b \geq \max\{s,(s-1)^{2s-2}/(s-1)!\}$.
Then, $$\left\lfloor ({(s-1)!b})^{1/{(2s-2)}} \right \rfloor - s+1 \leq \tau_{s-2}(s,b) \leq \left\lceil(b+1)^{1/s}\right\rceil-s.$$
\end{cor}

\begin{proof}
The lower bound follows immediately from Corollary~\ref{cor: its6not8}.
Note that $\left\lceil(b+1)^{1/s}\right\rceil-s \geq 0$ since $b \geq \max\{s,(s-1)^{2s-2}/(s-1)!\} \geq (s-1)^{s}$. Let $l = \lceil(b+1)^{1/s}\rceil-s$. Notice that $b \leq (l+s)^s-1$.
By Theorem~\ref{thm: half} we know that $\tau_{s-2}(s,b) \leq l$.
\end{proof}

We can use Corollary~\ref{thm: decent_sandwich} to improve Theorem~\ref{thm: Allagan} in the strongest possible sense.

\begin{customcor} {\bf \ref{cor: improveAllagan}}
For any $b \geq 2$, $\tau_0(2,b) = \lfloor\sqrt{b}\rfloor - 1.$
\end{customcor}

\begin{proof}
By Corollary~\ref{thm: decent_sandwich} we have that $\lfloor \sqrt{b} \rfloor - 1 \leq \tau_{0}(2,b) \leq \left \lceil \sqrt{b+1} \;\right\rceil-2.$ The result follows from $\left\lceil \sqrt{b+1}\; \right\rceil = \lfloor \sqrt{b} \rfloor +1$.
\end{proof}

A less important, but somewhat amusing application of Theorems~\ref{thm: its6not8} and~\ref{thm: half} is that $\tau_1(3,b) = 2$ whenever $96 \leq b \leq 124$ and $\tau_2(4,b) = 2$ whenever $1250 \leq b \leq 1295$.  In the following Lemma we prove a result that we will use in order to specifically study when $\tau_{s-2}(s,b) = 1$.

\begin{lem}\label{lem: baseCases:(}
    Suppose $s \geq 2$, and $G=K_{1,s}$. Also suppose that $\{z\}$, $\{x_1, \ldots, x_s\}$ is the bipartition of $G$. Suppose $L$ is an $(s+1)$-assignment for $G$ such that $L(x_1), \ldots , L(x_s)$ are pairwise disjoint. Then $$|\left\{f(V(G)):\text{$f$ is a proper $L$-coloring of $G$}\right\}| \geq \binom{s+1}{2}(s+1)^{s-1}.$$
\end{lem}

\begin{proof}
Let $\mathcal{L}$ be the set of proper $L$-colorings of $G$, and let $\mathcal{R} = \mathcal{R}(G,L)$ (i.e., $\mathcal{R} = \left\{f(V(G)):\text{$f$ is a proper $L$-coloring of $G$}\right\}$).  Note that each element of $\mathcal{R}$ has $s+1$ elements.  Let $F: \mathcal{L} \rightarrow \mathcal{R}$ be the function given by $F(f) = f(V(G))$ (i.e., $F$ maps each proper $L$-coloring of $G$ to its range).  We claim that for each $R \in \mathcal{R}$, $|F^{-1}(R)| \leq 2$.  To see why this is so, suppose $g$, $h$, and $p$ are pairwise distinct elements of $\mathcal{L}$ satisfying $g(V(G))=h(V(G))=p(V(G))$.  Notice that if $g(z) = h(z)$, it must be that $g=h$ since the fact that $L(x_1), \ldots , L(x_s)$ are pairwise disjoint implies that there can only be one way to assign the $s$ colors in $g(V(G)) - \{g(z)\}$ to the vertices in $\{x_1, \ldots, x_s\}$ in such a way that that vertex $x_i$ is assigned a color from $L(x_i)$ for each $i \in [s]$.  Consequently, $g(z)$, $h(z)$, and $p(z)$ are pairwise disjoint, and we let $a=g(z)$, $b=h(z)$, and $c=p(z)$. 

Since $g(V(G))=h(V(G))$, we know that $h$ must use the color $a$ on some vertex in $\{x_1, \ldots, x_s\}$.  Without loss of generality, suppose that $h(x_1)=a$, and consequently, $a \in L(x_1)$.  We claim that $g(x_1)=b$.  To see why this is so, suppose $g(x_1) \neq b$.  Then, $g(x_1) = d$ for some $d \in L(x_1)$ satisfying $d \neq a$ and $d \neq b$.  However, since $g(V(G))=h(V(G))$, $h$ would need to use the color $d$ on some vertex in $\{x_2, \ldots, x_s\}$ which contradicts the fact that $L(x_1), \ldots , L(x_s)$ are pairwise disjoint.  Thus, $g(x_1) = b$, and consequently, $b \in L(x_1)$.  Finally, note that $p$ must use $a$ and $b$ to color two of the vertices in $\{x_1, \ldots, x_s\}$ which again contradicts the fact that $L(x_1), \ldots , L(x_s)$ are pairwise disjoint.  Having obtained a contradiction, we have that $|F^{-1}(R)| \leq 2$ for each $R \in \mathcal{R}$.

Now, we will prove a lower bound on $|\mathcal{L}|$ which will allow us to complete the proof of the lemma.  Note that any coloring that colors $z$ with some $c_z \in L(z)$ can be extended to a proper $L$-coloring of $G$ in at least $s(s+1)^{s-1}$ ways since $c_z$ is in at most one of the lists $L(x_1), \ldots , L(x_s)$.  Consequently,
$$|\mathcal{L}| \geq (s+1)s(s+1)^{s-1} = 2 \binom{s+1}{2} (s+1)^{s-1}.$$
Finally,
$$ 2 \binom{s+1}{2} (s+1)^{s-1} \leq |\mathcal{L}| = \sum_{R \in \mathcal{R}} |F^{-1}(R)| \leq 2 |\mathcal{R}|$$
which implies the desired inequality.    
\end{proof}

We are finally ready to determine precisely when $\tau_{s-2}(s,b)$ equals 1 which will indicate that Theorem~\ref{thm: its6not8} is best possible when $l=2$.

\begin{customthm} {\bf \ref{cor: l=1}}
For each integer $s$ satisfying $s \geq 2$, $\tau_{s-2}(s,b) = 1$ if and only if $s^s \leq b \leq \binom{s+1}{2}(s+1)^{s-1}-1$.
\end{customthm}

\begin{proof}
Let $G = K_1 \vee K_{s,b}$.
By Lemma~\ref{lem: closeaandi} and Theorem~\ref{thm: its6not8} we know that $b \geq s^s$ if and only if $\tau_{s-2}(s,b) \geq 1$.

Suppose $\tau_{s-2}(s,b)=1$.  This means that $b \geq s^s$.
Let the vertex set of the copy of $K_1$ used to construct $G$ be $W = \{w_1\}$. 
Furthermore, suppose the copy of $K_{s,b}$ used to construct $G$ has bipartition $X = \{x_1,\ldots,x_s\}$, $Y = \{y_1, \ldots , y_b\}$.  
For the sake of contradiction, suppose that $b \geq \binom{s+1}{2}(s+1)^{s-1}$. Then Theorem~\ref{thm: its6not8} implies $\tau_{s-2}(s,b) \geq 2$ since $b \geq \binom{2+s-1}{2}(2+s-1)^{s-1}$ which is a contradiction.

Conversely, assume $s^s \leq b \leq \binom{s+1}{2}(s+1)^{s-1} - 1$.  This means $\tau_{s-2}(s,b) \geq 1$.
To show that $\tau_{s-2}(s,b) \leq 1$, we will show that $\chi_\ell(G)-\chi(G) \leq s-2$ which is equivalent to $\chi_\ell(G) \leq s+1$.
To show that $\chi_\ell(G) \leq s+1$ we will prove $G$ is $(s+1)$-choosable.
Suppose $L$ is an arbitrary $(s+1)$-assignment for $G$. 
We must show that there exists a proper $L$-coloring of $G$.
Note that if there exists $i,j \in [s]$ with $i \neq j$ and $L(x_i) \cap L(x_j) \neq \emptyset$ the existence of a proper $L$-coloring of $G$ follows from Lemma~\ref{lem: non-empty_General}. So we may assume that $L(x_1), \ldots, L(x_s)$ are pairwise disjoint.
Let $L'$ be the function $L$ with domain restricted to $W \cup X$.
By Lemma~\ref{lem: baseCases:(} we know that 
$|\mathcal{R}(G[W \cup X],L')| \geq \binom{s+1}{2}(s+1)^{s-1}$.
Since $b \leq \binom{s+1}{2}(s+1)^{s-1}-1$, the above inequality implies that there is a proper $L'$-coloring $g$ of $G[W \cup X]$ such that $L(y_i) \neq g(W \cup X)$ for each $i \in [b]$. Color the vertices in $W \cup X$ according to $g$. 
Let $L''(y_i) = L(y_i)-g(W \cup X)$ for each $i \in [b]$. Clearly $|L''(y_i)| \geq 1$ for each $i \in [b]$.
Finally, color $y_i$ with any color from $L^{\prime\prime}(y_i)$ for each $i\in[b]$.
This completes a proper $L$-coloring of $G$.
\end{proof}

\noindent{\bf Acknowledgment:} The authors would like to thank Hemanshu Kaul for his helpful comments.  The authors also thank the anonymous referee whose comments greatly improved this paper.  This project was completed as part of an undergraduate research course at the College of Lake County during the summer 2022, fall 2022, spring 2023, and summer 2023 semesters.  The support of the College of Lake County is gratefully acknowledged.

\end{document}